\documentclass[12pt]{article}
\usepackage{amsmath,amssymb,amsthm}
\theoremstyle{remark}
\newtheorem{proposition}{Proposition}[section]
\newtheorem{remark}{Remark}[section]

\theoremstyle{remark}

\numberwithin{equation}{section}
\usepackage{mdframed}
\usepackage{graphicx}
\usepackage{hyperref}

\usepackage{float}
\usepackage{subcaption}
\newtheorem{definition}{Definition}[section]

\title{Tube Integrability in a Time-Dependent Nonlinear Oscillator}

\author{Johannes Hagel \\[0.3em]
	 Alexander-von-Humboldt-Gymnasium Neuss \\	 
	 Bergheimer Straße 233, 41464 Neuss, Germany \\
	 \small johannes.hagel@gmail.com \\
	 \vspace{-0.5cm}
	 \date{1 November 2025}}

\begin{document}

\maketitle
\vspace{-1cm}
\begin{abstract}
	We study the nonlinear oscillator $z'' + \omega^2 z + g(t) z^2 = 0$ with a 
	time–dependent coefficient $g(t)$. We show that this equation admits an exact 
	quadratic invariant $I(z,p,t)$ provided that $g(t)=\alpha_2(t)^{-5/2}$ and 
	$\alpha_2(t)$ satisfies a nonlinear third–order differential equation. The 
	resulting invariant constrains the dynamics to a smooth two–dimensional 
	surface in the extended phase space $(z,p,t)$. If $\alpha_2(t)$ is periodic, 
	this surface forms a compact invariant torus. However, we prove that periodic 
	solutions of $\alpha_2(t)$ are generically obstructed by a resonance mechanism, 
	leading instead to an aperiodic but non–chaotic evolution. In this regime the 
	invariant surface is non–compact and extends along the time direction, forming 
	a \emph{tube} rather than a torus. We therefore propose the term 
	\emph{tube integrability} for integrable systems whose invariant manifolds 
	are non–compact in time. A perturbation expansion for $\alpha_2(t)$ up to 
	third order is derived and compared with numerical integration, clarifying 
	the parameter regimes in which the truncated series provides quantitatively 
	accurate approximations. The breakdown of the series for small $y_0$ is shown 
	to reflect the asymptotic nature of the expansion rather than a loss of 
	integrability.
\end{abstract}

\vspace{0.5cm}
\noindent


\section{Introduction}

Classical Liouville integrability is commonly associated with the geometry of 
invariant tori in phase space~\cite{Arnold}. In this classical setting, compactness 
is essential: it permits the introduction of globally defined angle variables and 
ensures recurrent motion. When a system is explicitly time-dependent, however, 
the connection between integrability and compact invariant geometry no longer follows.

In this work we study the nonlinear oscillator
\begin{equation}
	z'' + \omega^2 z + g(t)z^2 = 0 \label{eq:z_basic}
\end{equation}
where $g(t)$ varies smoothly in time. We show that this system admits an exact 
quadratic invariant $I(z,p,t)$, structurally related to the Ermakov--Lewis framework 
for time-dependent oscillators~\cite{Lewis1967,Ermakov1880,Pinney1950}. The existence 
of this invariant confines solutions to a smooth two-dimensional surface in the 
extended phase space $(z,p,t)$.

If $\alpha_2(t)$ is periodic, Floquet theory implies that these invariant surfaces 
form compact tori~\cite{Floquet1883,Hill1886}. However, we prove that periodic 
solutions of $\alpha_2(t)$ are generically obstructed by a second-order resonance 
mechanism, analogous to small-divisor effects in classical perturbation theory 
and averaging methods~\cite{Lindstedt,SandersVerhulst}. As a consequence, the 
generic behavior of $\alpha_2(t)$ is aperiodic.

In this non-periodic regime the invariant surface does not close on itself, but 
forms a smooth, non-compact structure extending along the time axis. We therefore 
propose the term \emph{tube integrability} to describe integrable systems whose 
invariant manifolds are non-compact in time. 

The perturbation series for $\alpha_2(t)$ is shown to be asymptotic rather than 
convergent, with the eventual breakdown of truncations governed by nearby complex 
singularities in the continuation parameter~\cite{BenderOrszag,Olver,Berry1989}. 
This breakdown does not indicate a loss of integrability, but marks the natural 
limit of the truncated series.

\medskip
\noindent\textbf{Remark.}
The conceptual motivation for this work originates in the long-standing question
of whether integrability is intrinsically tied to compact invariant geometry.
The results presented here demonstrate that this need not be the case: exact
integrability may coexist with non-compact invariant manifolds, provided the
geometric structure is coherent and the invariant is global.

\section{General formulation and derivation of the invariant}

We start from the quadratic ansatz for a first integral
\begin{equation}
I(z,p,t) = a_0(z,t) + a_1(z,t) p + a_2(t) p^2,
\qquad p = z',
\end{equation}
which was derived in ~\cite{hagel-bouquet-1992} by requiring $\tfrac{dI}{dt}=0$ for all trajectories of
\begin{equation}
z'' + \omega^2 z + g(t) z^2 = 0. \label{osc_non_lin}
\end{equation}

As shown previously, the general solution for the coefficient functions is
\begin{align}
a_2(t) &= \alpha_2(t), \\
a_1(z,t) &= -\alpha_2'(t) z + \alpha_1(t), \\
a_0(z,t) &= \omega^2 \alpha_2(t) z^2 + \tfrac{2}{3} \alpha_2(t) g(t) z^3 - \alpha_1'(t) z + \tfrac{1}{2} \alpha_2''(t) z^2 + \alpha_0(t).
\end{align}

Inserting these expressions into the determining equations yields the following coupled system:
\begin{align}
\alpha_1'' + \omega^2 \alpha_1 &= 0, \tag{2.1} \\
\alpha_2''' + 4 \omega^2 \alpha_2' - 2 \alpha_1 g(t) &= 0, \tag{2.2} \\
5 g(t) \alpha_2' + 2 \alpha_2 g'(t) &= 0, \tag{2.3} \\
g(t) &= \alpha_2(t)^{-5/2}. \tag{2.4}
\end{align}

Equation (2.1) admits the oscillatory solution
\begin{equation}
\alpha_1(t) = \tfrac{1}{2} C_1 \cos(\omega t) + \tfrac{1}{2} C_2 \sin(\omega t), \tag{2.5}
\end{equation}
so that substitution into (2.2) yields the nonlinear third-order ODE
\begin{equation}
\alpha_2''' + 4\omega^2 \alpha_2' - \big[ C_1 \cos(\omega t) + C_2 \sin(\omega t)\big] \alpha_2^{-5/2} = 0. \tag{2.6}
\end{equation}

Taken together, (2.1)--(2.4) form the set of conditions for integrability in the quadratic case with nonzero $\alpha_1(t)$.

Finally, the quadratic invariant takes the explicit form
\begin{align}
I(z,p,t) &= \alpha_2(t) p^2 - \alpha_2'(t) z p + \alpha_1(t) p \nonumber \\
&\quad + \Big( \omega^2 \alpha_2(t) + \tfrac{1}{2} \alpha_2''(t) \Big) z^2
 - \alpha_1'(t) z + \tfrac{2}{3} \alpha_2(t) g(t) z^3 \;=\; K. \tag{2.7}
\end{align}

\bigskip
\newpage
\noindent\textbf{Summary of the integrable system:}
\begin{mdframed}[linewidth=1pt]
\begin{align*}
(S1)\;& z'' + \omega^2 z + g(t) z^2 = 0 \\[0.5em]
(S2)\;& g(t) = \alpha_2(t)^{-5/2} \\[0.5em]
(S3)\;& \alpha_2''' + 4\omega^2 \alpha_2' - [C_1\cos(\omega t)+C_2\sin(\omega t)]\,\alpha_2^{-5/2} = 0 \\[0.5em]
(S4)\;& I(z,p,t) = \alpha_2(t)p^2 - \alpha_2'(t)zp + \alpha_1(t)p \\
     &\quad + \Big(\omega^2\alpha_2(t)+\tfrac{1}{2}\alpha_2''(t)\Big)z^2
       - \alpha_1'(t) z + \tfrac{2}{3}\alpha_2(t) g(t) z^3 \;=\; K
\end{align*}
\end{mdframed}
\bigskip

\section{Tube integrability}
From Eqs. (S3) and (S4) we see, that  periodic solutions of $\alpha_2(t)$ lead to classical invariant tori 
in the extended phase space while non-periodic solutions require a different 
concept. In this general case we introduce the notion of \emph{tube integrability} 
which extends the classical concept of Liouville integrability to certain non-periodic situations.

\begin{definition}[Tube integrability]
	We denote a system
	\[
	z'' + \omega^2 z + g(t)\,z^2 = 0
	\]
	as \emph{tube integrable} if there exists a polynomial invariant
	\[
	I(z,p,t) = \sum_{k=0}^N a_k(t)\,p^k,
	\qquad p=z',
	\]
	with smooth time-dependent coefficients $a_k(t)$, such that
	\(\tfrac{d}{dt} I(z(t),p(t),t) \equiv 0\),
	even though the coefficient functions $a_k(t)$ are not periodic in~$t$.
	
	In this case, the flow $(z(t),p(t),t)$ is restricted to a smooth two-dimensional invariant surface in the extended phase space $(z,p,t)$. 
	Since $a_k(t)$ are not periodic, a Poincaré section at times $t=nT$ is formally admissible but does not produce closed invariant curves in the $(z,p)$-plane. Instead, the motion is confined to an \emph{invariant tube}, a non-periodic deformation of a classical invariant torus.
\begin{figure}[H]
	\centering
	\includegraphics[width=0.7\textwidth]{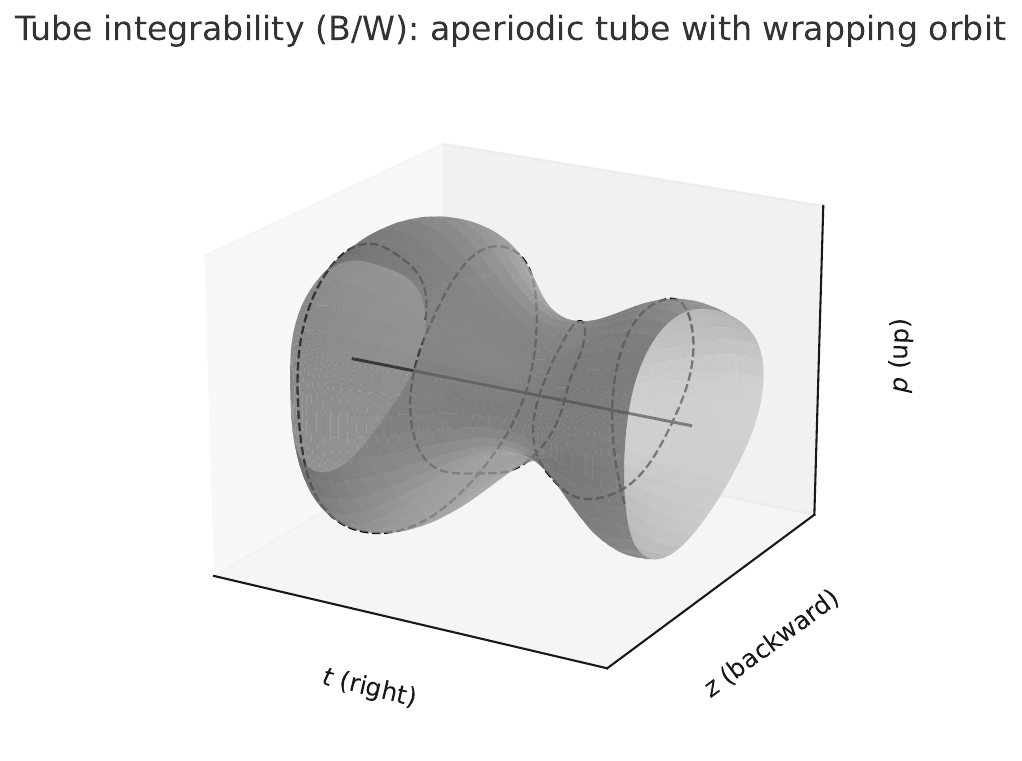}
	\caption{Schematic illustration of tube integrability. 
		The $t$–axis (right) runs inside the tube, while the cross–section
		$(z,p)$ varies aperiodically so that no invariant torus arises. 
		Instead, the solution winds around the $t$–axis on a non-periodic
		tube surface.}
	\label{fig:tube-schematic}
\end{figure}

\noindent Note that tube integrability is a property of the original $z$--equation (\ref{eq:z_basic}). The auxiliary function $y(\tau)=\alpha_2(t)$ appears only through the time-dependent coefficient $g(t)=y^{-5/2}$ which determines the geometry of the invariant surface. Thus, periodic $y$ would imply torus integrability of the $z$--equation, while aperiodic $y$ yields tube integrability of the $z$--equation.
\end{definition}
\begin{remark}
	Classical Liouville integrability with periodic coefficients $a_k(t)$
	leads to invariant two-dimensional tori in the extended phase space
	and to closed invariant curves in Poincaré sections.
	In contrast, tube integrability describes the situation where the
	invariant surface is topologically a ``tube'' rather than a torus:
	the trajectory winds quasiperiodically or even irregularly along the
	time direction, without closing. 
	From the point of view of $(z,p)$ only, the motion therefore appears
	irregular and does not admit classical Poincaré invariant curves,
	even though an exact invariant $I(z,p,t)$ still exists.

\noindent Consequently, the lack of periodicity in $\alpha_2(t)$ is directly transferred to the $z$--equation: no invariant tori arise, but the invariant surfaces are of tube type.
\end{remark}
As an explicit illustration of tube integrability, we present in
Fig.~\ref{fig:tube-integrability} a numerical example with non\-periodic
$\alpha_2(t)$. The four panels display $\alpha_2(t)$, the corresponding
coefficient $g(t)=\alpha_2(t)^{-5/2}$, the resulting solution $z(t)$,
and the quadratic invariant $I(t)$.
	\begin{figure}[H]
		\centering
	\begin{subfigure}[t]{0.48\linewidth}
		\includegraphics[width=\linewidth]{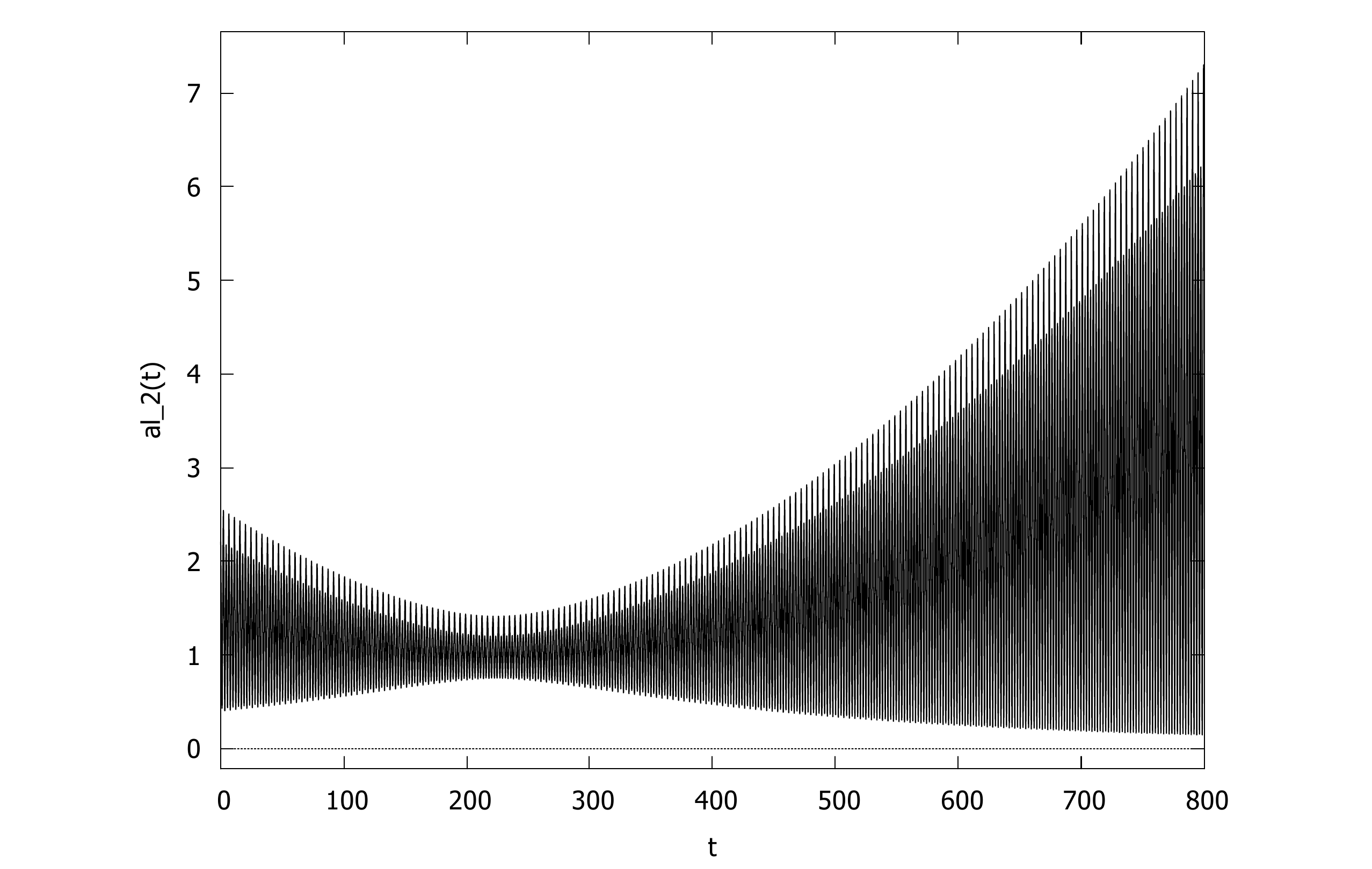} 
		\caption{$\alpha_2(t)$}
	\end{subfigure}
	\hfill
	\begin{subfigure}[t]{0.48\linewidth}
		\includegraphics[width=\linewidth]{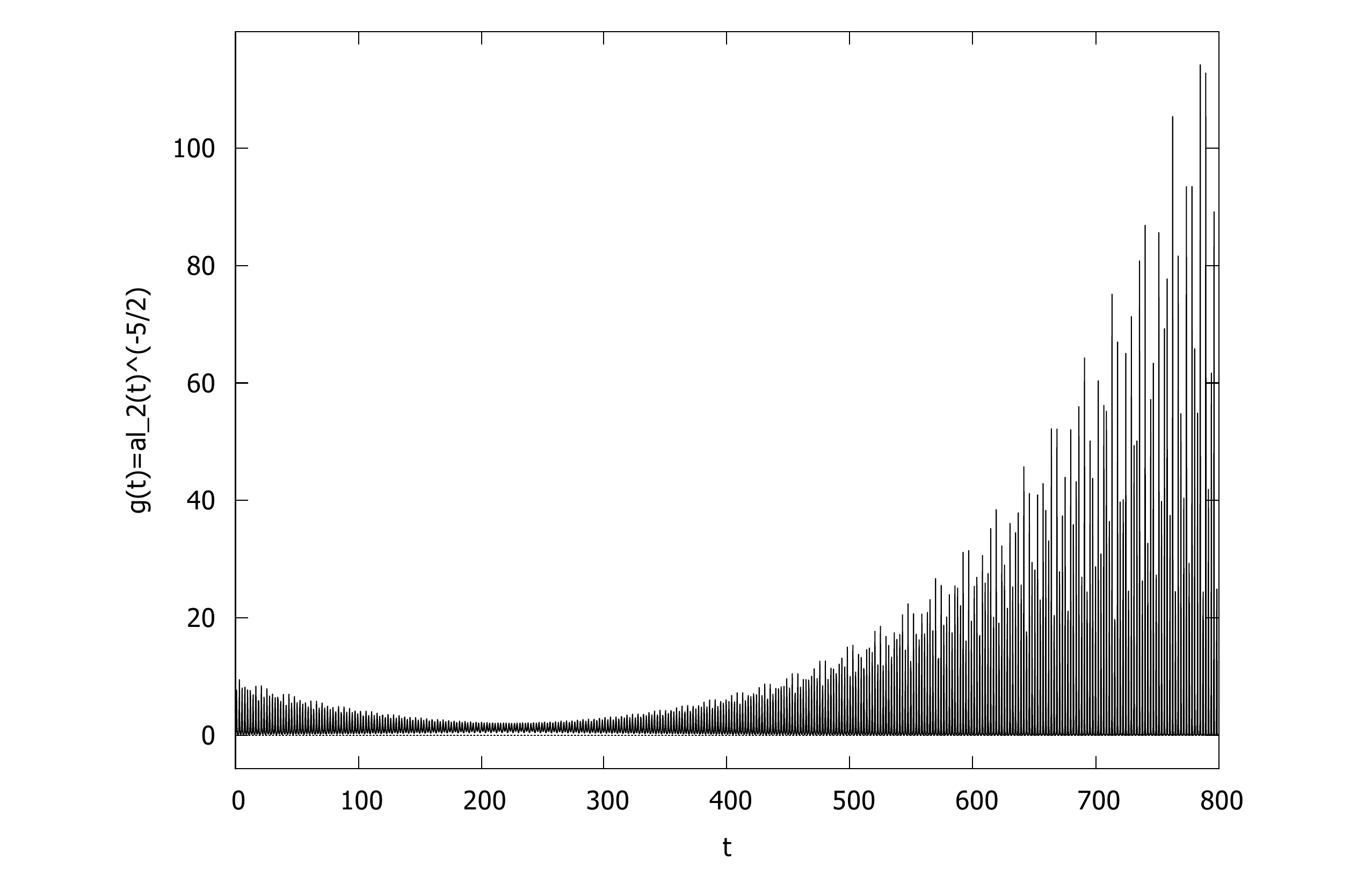} 
		\caption{$g(t)=\alpha_2(t)^{-5/2}$}
	\end{subfigure}
	
	\vspace{0.5cm}
	
	\begin{subfigure}[t]{0.48\linewidth}
		\includegraphics[width=\linewidth]{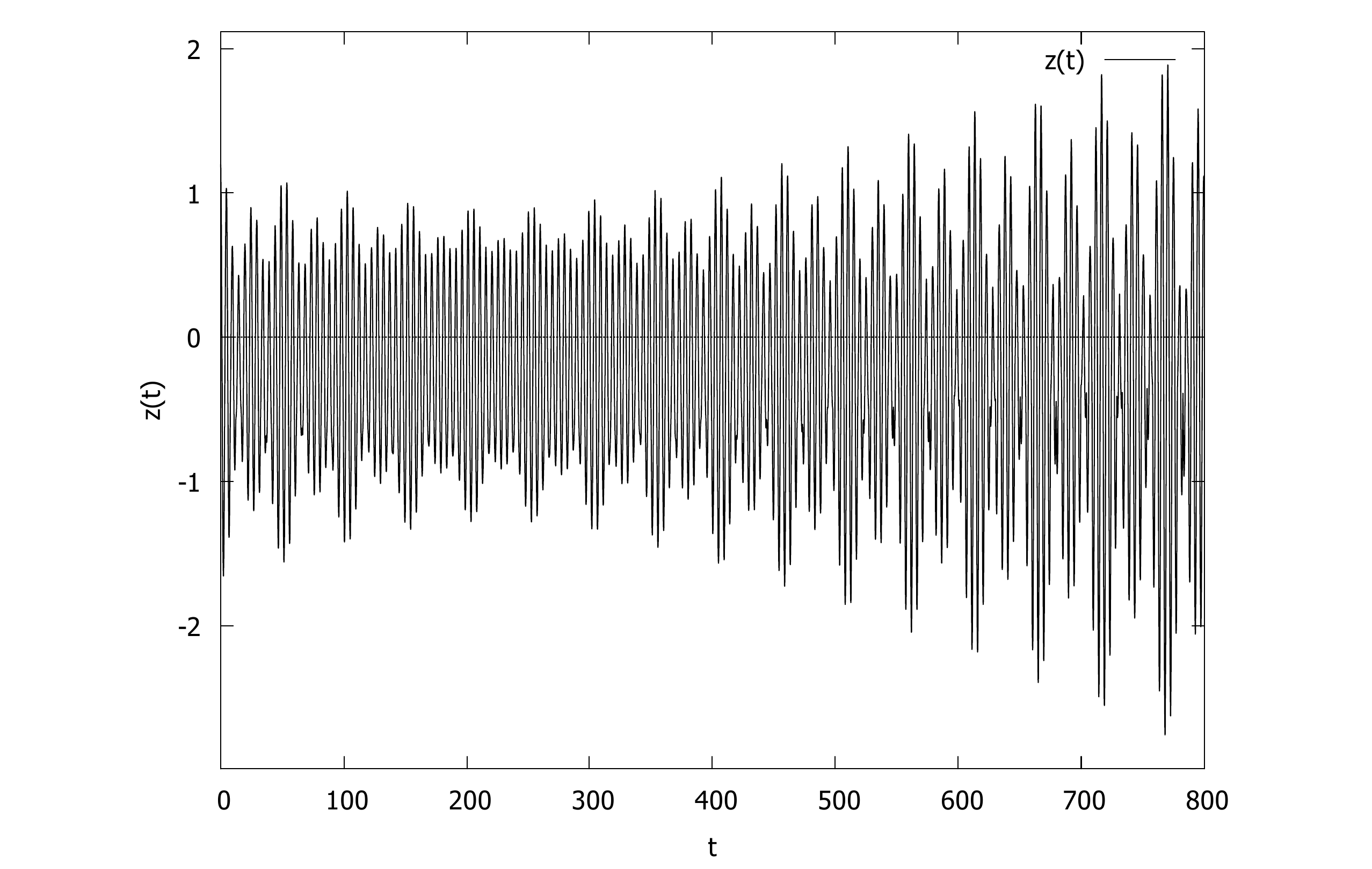} 
		\caption{$z(t)$}
	\end{subfigure}
	\hfill
	\begin{subfigure}[t]{0.48\linewidth}
		\includegraphics[width=\linewidth]{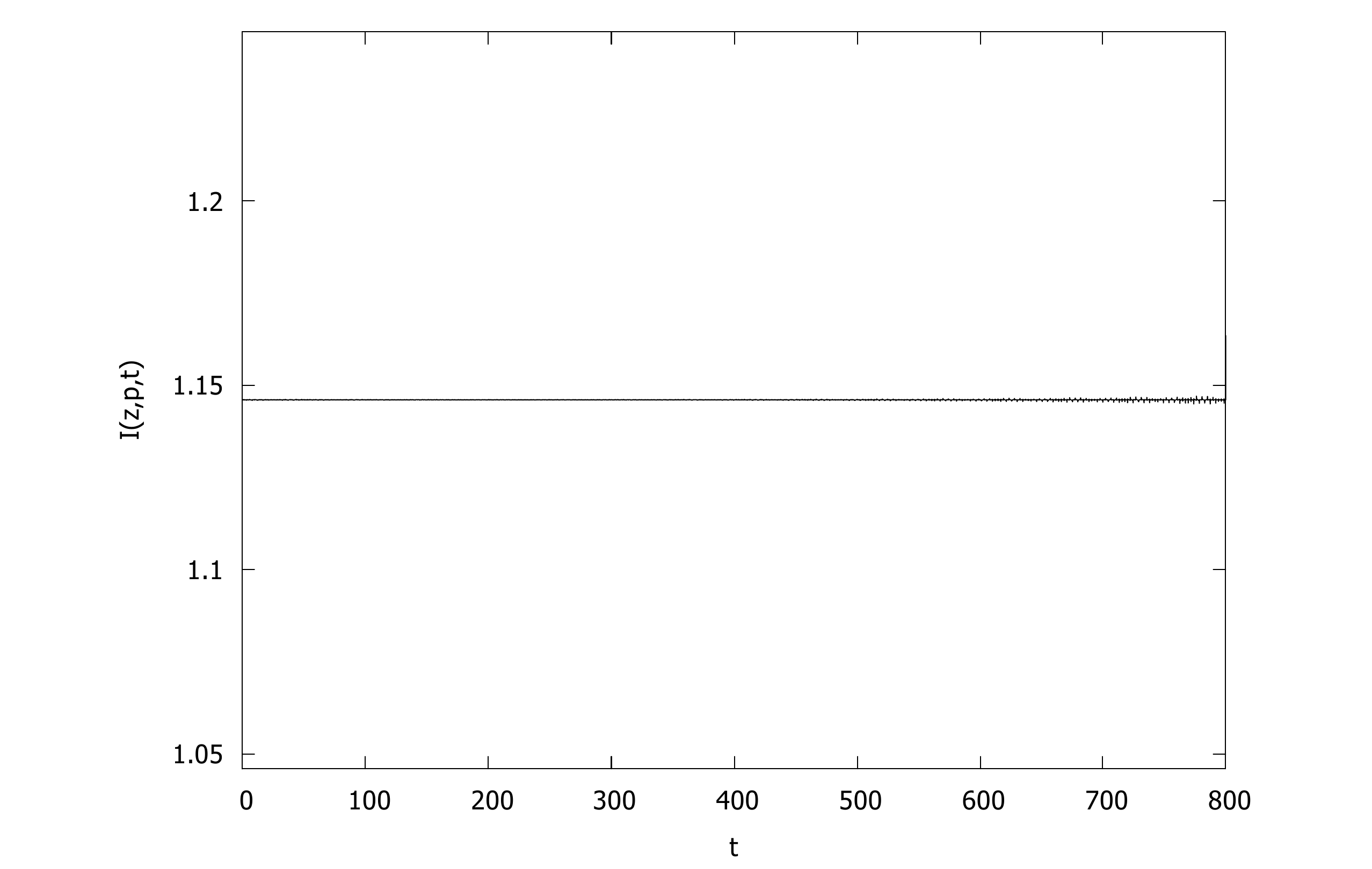} 
		\caption{Quadratic invariant $I(t)$}
	\end{subfigure}
	
	\caption{Illustration of \emph{tube integrability} in the case of a nonperiodic
		$\alpha_2(t)$. The function $\alpha_2(t)$ shows dominant frequencies $\omega$ and $2\omega$
		together with a slowly modulated envelope. The coefficient $g(t)$ injects energy,
		leading to irregular but bounded oscillations of $z(t)$. Nevertheless the quadratic
		invariant $I(t)$ remains conserved up to numerical accuracy, confirming motion on
		an invariant tube.}
	\label{fig:tube-integrability}

\end{figure}
\subsection{Example: tube integrability via the Ermakov–Pinney construction}
As an illustrating example of tube integrability in the sense of the above definition, we
consider the linear oscillator
\begin{equation}\label{eq:lin-osc}
	z'' + f(t)\,z = 0,
\end{equation}
and let $w(t)$ be any strictly positive solution of the Ermakov–Pinney equation
\begin{equation}\label{eq:EP}
	w'' + f(t)\,w = \frac{1}{w^3}.
\end{equation}
Then the Lewis–Ermakov invariant
\begin{equation}\label{eq:Lewis}
	I(z,p,t) \;=\; \tfrac12\Big(\,\big(\tfrac{z}{w}\big)^2 \;+\; \big(w\,p - w'\,z\big)^2\,\Big),
	\qquad p=z',
\end{equation}
satisfies $\tfrac{d}{dt}I(z(t),p(t),t)\equiv 0$ along all solutions of \eqref{eq:lin-osc}, for \emph{any} given coefficient $f(t)$.

\begin{proof}[Sketch]
	A direct differentiation of \eqref{eq:Lewis} and substitution of $z''=-f(t)z$ and $w''=-f(t)w+1/w^3$ yields $\dot I\equiv0$.
\end{proof}

\paragraph{Tube (vs. torus) behavior.}
If $f(t)$ is $T$–periodic, the coefficients in \eqref{eq:Lewis} are $T$–periodic and classical Liouville integrability manifests in invariant tori (closed invariant curves in Poincaré sections).
If, however, $f(t)$ is nonperiodic—e.g.\ obtained from a chaotic driver—then $I$ still defines an exact invariant surface in the extended space $(z,p,t)$, but the surface winds nonperiodically in $t$, hence an \emph{invariant tube} rather than a torus; classical Poincaré sections in $(z,p)$ do not yield closed invariant curves.

\paragraph{Chaotic forcing example (logistic driver).}
Let $\{\ell_n\}$ be the logistic map at $\mu=4$, $\ell_{n+1}=4\ell_n(1-\ell_n)$ on $[0,1]$.
Choose a sampling period $T_s>0$ and define a smooth $C^2$ interpolation $\tilde f(t)$ such that
$\tilde f(nT_s)=f_0+\Delta f\,(2\ell_n-1)$ for parameters $f_0>0$, $\Delta f\in(0,f_0)$.
(For instance, use cubic splines on the grid $t=nT_s$.)
Set $f(t)=\tilde f(t)$ in \eqref{eq:lin-osc}–\eqref{eq:EP}.
Then $f(t)$ is aperiodic (chaotic in the sense of the underlying symbolic dynamics), while the invariant \eqref{eq:Lewis} remains exact.
Consequently, the motion of \eqref{eq:lin-osc} is confined to an invariant tube defined by $I=\mathrm{const}$, i.e.\ the system is \emph{tube integrable}.

\section{Properties and computation of the function $\alpha_2(t)$}
As has been shown in section 2, the function $\alpha_2(t)$ given by the third order nonlinear differential equation (S3) via its solutions generates an infinite number of integrable systems for the associated $z$ equation (\ref{eq:z_basic}). In this section we discuss the properties of $\alpha_2(t)$ as well as perturbative approaches to this explicitly time dependent nonlinear equation.

\subsection{Transformation of the independent variable}
We first rewrite the dependent variable $\alpha_2$ replacing it by $y$ and start from the original equation
\begin{equation}
	y'''+4\omega^2 y'-[C_1 \cos{\omega t}+C_2 \sin{\omega t}]y^{-\frac{5}{2}}=0 \label{eq:y(t)}
\end{equation} 
Next, we simplify the equation, rescaling the independent variable as
\begin{equation}
	\tau=\omega t \label{eq:tau}
\end{equation}
leading to
\begin{equation}
	y'''+4y'-\frac{1}{\omega^3}[C_1 \cos{\tau}+C_2 \sin{\tau}]y^{-\frac{5}{2}}=0 \label{eq:y_red}
\end{equation}
where from this point ' denotes the derivative w.r.t. $\tau$. We consider all the variables and constants included in (\ref{eq:y_red}) to be real numbers.

\subsection{Positivity of $y(\tau)$}
We prove the important Lemma: 

\medskip

\noindent Lemma: All solutions of (\ref{eq:y_red}) are strictly positive functions, $y(\tau) > 0$ 

\medskip

\noindent Proof by contradiction: We rewrite (\ref{eq:y_red}) as
\begin{equation}
	y'''+4y'= \frac{1}{\omega^3}[C_1 \cos{\tau}+C_2 \sin{\tau}]y^{-\frac{5}{2}}
\end{equation}
Let's consider the existance of an interval $(\tau_1,\tau_2)$ in which $y<0$. Then the left side of the above equation remains a real function. However, the right hand side of the equation, due to the power $-\frac{5}{2}$, becomes a pure imaginary function which makes it impossible to satisfy the differential equation. In order to avoid the contradiction, we conclude $y>0$ for all possible solutions of (\ref{eq:y_red}). q.e.d.
\medskip

\subsection{Reduction of the non autonomous term to a single cosine function}

Since the nonautonomous part of (\ref{eq:y_red}) can be written as a cosine function in the form:
\begin{equation}
	C_1 \cos{\tau}+C_2 \sin{\tau} =\sqrt{C_1^2+C_2^2}\cos(\tau-\Phi)
\end{equation}
we may, without loss of generality, rewrite the y-equation as
\begin{equation}
	y'''+4y'=\varepsilon \cos{\tau} y^{-\frac{5}{2}} \label{eq:y_simp}
\end{equation} 
by using 
\begin{equation}
	\varepsilon=\frac{C}{\omega^3} \; \; ; \; \; C=\sqrt{C_1^2+C_2^2}
\end{equation}
as well as applying a linear constant shift of $\Phi$ to the independent variable.
\subsection{Concerning the existence of a periodic solution for $y(\tau)$}
\label{sec:alpha2-periodicity}

We now address the question whether the function $y(\tau)$
can admit strictly periodic
solutions. As has been shown, $y(\tau)$ satisfies
\begin{equation}\label{eq:alpha2-eq}
	y''' + 4 y' + \varepsilon y^{-5/2}\cos \tau = 0, 
	\qquad y(\tau) = \alpha_2(t), \quad \varepsilon=\frac{C}{\omega^3},
\end{equation}
with $\tau=\omega t$. 

\begin{proposition}[Non-existence of small $2\pi$--periodic solutions]
	\label{prop:no-periodic-alpha2}
	Let $m>0$ and suppose that $y(\tau)=\alpha_2(t)$ is a $2\pi$--periodic solution of
	\eqref{eq:alpha2-eq} with mean value
	\[
	m = \tfrac{1}{2\pi}\int_0^{2\pi} y(\tau)\, d\tau.
	\]
	Then for $\varepsilon \neq 0$ and $C_1\neq 0$ no such solution bifurcates
	from the constant state $y\equiv m$.
\end{proposition}

\begin{proof}[Sketch of proof]
	Write $y(\tau)=m+u(\tau)$ with $\int_0^{2\pi} u(\tau)\, d\tau=0$.
	Linearization gives the operator $L = D^3+4D$ on $2\pi$--periodic functions,
	with kernel
	\[
	\ker L = \mathrm{span}\{1,\sin(2\tau),\cos(2\tau)\}.
	\]
	At order $\mathcal O(\varepsilon)$ one finds a $2\pi$--periodic correction
	$u_1(\tau)=a_1\sin\tau$. At order $\mathcal O(\varepsilon^2)$, however,
	the nonlinear expansion of $y^{-5/2}$ generates a forcing proportional
	to $\sin(2\tau)$, which lies inside $\ker L$. By Fredholm’s alternative,
	solvability requires this resonant component to vanish, but its coefficient
	$\propto \varepsilon^2 m^{-6}$ is nonzero for $\varepsilon \neq 0$. Hence
	the solvability condition fails and no small $2\pi$--periodic solution can
	exist near $y\equiv m$.
\end{proof}

\begin{remark}\label{rem:alpha2-aperiodic}
	The resonance obstruction is fundamental. Since the forcing $\cos\tau$
	fixes the external frequency, Poincar\'e--Lindstedt detuning is not available
	in this non-autonomous setting. The consequences are:
	\begin{itemize}
		\item Up to first order in $\varepsilon$ one finds a strictly
		$2\pi$--periodic approximation of $\alpha_2(t)$, so the solution appears
		``almost periodic'' on short time scales.
		\item At second order the resonant $\sin(2\tau)$ contribution produces
		unavoidable secular terms, destroying exact periodicity.
		\item For $C_1=0$ the constant solution remains valid. For $C_1\neq 0$,
		$\alpha_2(t)$ is never exactly $2\pi$--periodic: it is at best
		\emph{almost periodic}, but in fact genuinely aperiodic.
		\item As a consequence, the invariant $I(z,p,t)$ of the $z$--equation
		still exists, but induces only \emph{tube integrability} in the extended
		phase space $(z,p,t)$, not a genuine invariant torus.
	\end{itemize}
\end{remark}
\subsection{Fourier expansion and resonance mechanism}
\label{sec:fourier}
For completeness we display the Fourier system corresponding to
equation~\eqref{eq:alpha2-eq}. We expand
\begin{equation}\label{eq:fourier-ansatz}
	y(\tau) = c_0 + \sum_{n\ge 1}\big(c_n \cos(n\tau) + s_n \sin(n\tau)\big),
\end{equation}
with $c_0>0$ the mean value.

\subsubsection{First order}
At order $\mathcal O(\varepsilon)$ one obtains
\[
s_1 = -\frac{\varepsilon}{3}\,c_0^{-5/2}, \qquad 
c_n = s_n = 0 \ \ (n\ge 2).
\]
Thus the first correction is purely $\sin\tau$, consistent with the forcing
phase.

\subsubsection{Second order}

Expanding $y^{-5/2}$ produces quadratic terms of the form
\[
(c_0 + s_1\sin\tau)^{-5/2}
= c_0^{-5/2}\Bigl( 1 - \tfrac{5}{2}\tfrac{s_1}{c_0}\sin\tau
+ \tfrac{35}{8}\tfrac{s_1^2}{c_0^2}\sin^2\tau + \cdots \Bigr).
\]
The $\sin^2\tau$ term generates a second harmonic $\propto \sin(2\tau)$
with coefficient
\[
-\frac{5}{96}\,\varepsilon^2\, c_0^{-6}.
\]
Since $(D^3+4D)(\sin 2\tau)=0$, this contribution lies in the kernel of
the linear operator and produces a secular term $\propto \tau\sin(2\tau)$.
This illustrates explicitly why the Fourier system has no consistent
$2\pi$--periodic solution at order $\varepsilon^2$.

\subsubsection{Third order}
For completeness, the cubic term forces a third harmonic
\[
s_3 = -\frac{7}{864}\,\varepsilon^3\, c_0^{-19/2}, \qquad c_3=0.
\]
Higher harmonics can in principle be computed recursively, but the
inconsistency at order $\varepsilon^2$ already prevents exact periodicity.

\medskip
In summary, the Fourier system \eqref{eq:fourier-ansatz} confirms the
Proposition~\ref{prop:no-periodic-alpha2}: the unavoidable $\sin(2\tau)$
forcing at second order makes the system unsolvable in the class of
$2\pi$--periodic functions. The apparent periodicity at first order is
destroyed by resonance at higher order.

\section{Perturbation theoretic computation of \(y(\tau)\)}
Since (\ref{eq:y_simp}) is a nonlinear, explicitly time dependent differential equation of order 3, we may not expect to find closed analytical expressions to its solutions. It is for this reason, that we introduce a perturbative treatment in order to approach the equations behaviour in an analytical way. It turned out however, that we are left with simpler expressions if we first integrate the entire equation w.r.t. $\tau$, leading to a volterra type integro-differential equation:

\begin{equation}\label{eq:y-ode}
	y''(\tau) + 4\,y(\tau) \;=\; y_0''+4\,y_0 \;+\; \varepsilon \int_{0}^{\tau} y(s)^{-5/2}\,\cos s \, ds,
\end{equation}

As can be seen the solution of this equation depends on 4 free parameters, namely the initial conditions $y_0$,$y_0'$ and $y_0''$ as well as on the parameter $\varepsilon$ defining the strength of the nonlinear contribution in the integrand. However, in order to keep the computations straightforward as possible, we set $y_0'=y_0''=0$ and remain with just two independent parameters $y_0$ and $\varepsilon$. Under this restriction (\ref{eq:y-ode})  reduces to
\begin{equation}
	y''+4y=4y_0+\varepsilon \int_0^\tau y(s)^{-\frac{5}{2}} \cos{s} ds \label{eq:y_pert}
\end{equation}
Next we take care of the fact that $y(\tau)$
for all real values of the independent variable. In order to keep this property for the entire perturbation procedure, we use the subtitution:
\begin{equation}
	y(\tau)=e^{\rho(\tau)}  
\end{equation}
Under this substitution we get
\begin{equation}
	y'=y\rho'  ;  y''=y'\rho'+y\rho''=y\rho'^2+y\rho''
\end{equation}
Inserting the so obtained subtitutions into (\ref{eq:y_pert}) and dividing by $e^\rho$ gives the new integro-differential equation for $\rho(\tau)$ as:

\begin{equation}\label{eq:rho-master}
	\rho''+(\rho')^2+4 \;=\; 4y_0 e^{-\rho} \;+\; \varepsilon\,e^{-\rho}\!\int_0^{\tau} e^{-5\rho(s)/2}\cos s \, ds .
\end{equation}
We aim for a perturbation expansion up to $\varepsilon^2$, hence we let:
\begin{equation}
	\rho(\tau)=\rho_0+\varepsilon\rho_1(\tau))+\varepsilon\rho_2(\tau)+\mathcal O(\varepsilon^3)
\end{equation}

\subsection*{Order \(\mathcal O(1)\)}
Comparing terms independent of $\varepsilon$
from \eqref{eq:rho-master} we get \(4=4y_0 e^{-\rho_0}\), hence
\begin{equation}\label{eq:rho0}
	\boxed{\;\rho_0=\ln y_0.\;}
\end{equation}

\subsection*{Order \(\mathcal O(\varepsilon)\)}
Collecting \( \mathcal O(\varepsilon) \)-terms gives
\begin{equation}\label{eq:rho1-ode}
	\rho_1''+4\rho_1 \;=\; y_0^{-7/2}\,\sin \tau, 
	\qquad \rho_1(0)=\rho_1'(0)=0,
\end{equation}
so that
\begin{equation}\label{eq:rho1-sol}
	\boxed{\;
		\rho_1(\tau)\;=\;y_0^{-7/2}\!\left(\frac{1}{3}\sin\tau-\frac{1}{6}\sin 2\tau\right).
		\;}
\end{equation}

\subsection*{Order \(\mathcal O(\varepsilon^2)\)}
At second order one finds
\begin{equation}\label{eq:rho2-ode}
	\rho_2''+4\rho_2 \;=\; 
	2\rho_1^2-(\rho_1')^2 
	\;+\; y_0^{-7/2}\!\left[-\frac{5}{2}\!\int_0^{\tau}\!\rho_1(s)\cos s\,ds-\rho_1(\tau)\sin\tau\right] ,
\end{equation}
\begin{equation}
	\rho_2(0)=\rho_2'(0)=0. \nonumber
\end{equation}
Inserting \eqref{eq:rho1-sol} and evaluating the elementary integrals yields the explicit solution
\begin{equation} \boxed{
\begin{aligned}\label{eq:rho2-sol} 
			\rho_2(\tau)
			&=y_0^{-7}\Bigg[
			-\frac{5}{288}
			-\frac{1}{24}\cos\tau
			+\frac{19}{288}\cos 2\tau
			-\frac{1}{72}\cos 3\tau \\
			&+\frac{1}{144}\cos 4\tau 
			+\frac{5}{96}\,\tau\sin 2\tau
			\Bigg]. 
\end{aligned}
}
\end{equation}
The term \( \frac{5}{96}y_0^{-7}\,\tau\sin(2\tau) \) encodes the intrinsic second–order secular response at the \(2{:}1\) resonance.

\subsection*{Composite \(\mathcal O(\varepsilon^2)\) approximation for \(y(\tau)\)}
Expanding \(e^{\rho}\) to second order gives
\begin{equation}\label{eq:yO2}
	y(\tau)=e^{\rho(\tau)}
	= y_0\left[\,1+\varepsilon\rho_1(\tau)+\varepsilon^2\Big(\rho_2(\tau)+\tfrac{1}{2}\rho_1(\tau)^2\Big)\right]
	\;+\;\mathcal O(\varepsilon^3),
\end{equation}
with \(\rho_1,\rho_2\) from \eqref{eq:rho1-sol}–\eqref{eq:rho2-sol}.

\paragraph{Scaling remark.}
Writing \(y(\tau)=y_0\,\tilde y(\tau)\) and normalizing \eqref{eq:y-ode} shows that the effective small parameter is
\[
\varepsilon_{\mathrm{eff}}=\varepsilon\,y_0^{-7/2},
\]
which explains the strong improvement of accuracy for larger \(y_0\) and the deterioration for \(y_0\ll 1\).

\paragraph{Heuristic validity window.}
The dominant secular piece in \eqref{eq:yO2} scales like 
\(
\sim \frac{5}{96}\,\varepsilon^2\,y_0^{-6}\,\tau\sin(2\tau).
\)
A practical criterion is therefore
\begin{equation}
	\tau<\tau^*=\frac{96y_0^6}{5\varepsilon^2}
	\label{eq:validity}
\end{equation}

\subsection{Third-order contribution $\rho_{3}(\tau)$}

We recall the Volterra formulation
\begin{equation}
	\rho''(\tau)+4\,\rho(\tau)
	= -\frac{5}{2}\,\varepsilon\,e^{-\frac{5}{2}\rho(\tau)}\int_{0}^{\tau} e^{-\frac{5}{2}\rho(s)}\cos s\,ds,
\end{equation}
with $\rho(0)=\rho'(0)=0$.  
Expanding $\rho(\tau)=\sum_{n\ge1}\varepsilon^{n}\rho_{n}(\tau)$ and the RHS of (\eqref{eq:rho-master}) up to order $\varepsilon^{3}$ yields
\begin{equation}
	\rho_{3}''(\tau)+4\,\rho_{3}(\tau)=F_{3}(\tau),
\end{equation}
where the source term is
\begin{align}
	F_{3}(\tau)
	&=
	-\frac{5}{2}\,e^{-\frac{5}{2}\rho_{0}}
	\Bigg[
	-\frac{5}{2}\rho_{2}(\tau)
	+\frac{25}{8}\rho_{1}(\tau)^{2}
	\Bigg]
	\int_{0}^{\tau} e^{-\frac{5}{2}\rho_{0}}\cos s\,ds
	\notag\\
	&\qquad
	-\frac{5}{2}\,e^{-\frac{5}{2}\rho_{0}}
	\int_{0}^{\tau}
	\left[
	-\frac{5}{2}\rho_{2}(s)+\frac{25}{8}\rho_{1}(s)^{2}
	\right]
	e^{-\frac{5}{2}\rho_{0}}
	\cos s\,ds .
\end{align}
Here $\rho_{1}$ and $\rho_{2}$ are the previously obtained first and second order solutions.
The inhomogeneous equation for  with initial data $\rho_{3}(0)=\rho_{3}'(0)=0$ can easily be solved by comparing the coefficients of all trigonometric functions and of the appearing resonance terms. Finally we obtain:
	
\begin{equation}
\boxed{
\begin{aligned} \label{eq:rho3_sol}
\rho_{3}(\tau)
&=\frac{1}{2592\,y_{0}^{21/2}}\Big[ 
-\frac{45}{4}\,\tau
+\frac{135}{4}\,\tau\cos\tau
-\frac{45}{2}\,\tau\cos(2\tau)
+\frac{45}{4}\,\tau\cos(3\tau) \\
&-\frac{45}{4}\,\tau\cos(4\tau) 
+\frac{159}{2}\,\sin\tau
-\frac{327}{4}\,\sin(2\tau)
+\frac{73}{4}\,\sin(3\tau)
+\frac{69}{8}\,\sin(4\tau) \\
&-\frac{9}{4}\,\sin(5\tau)
+\sin(6\tau)
\Big].
\end{aligned} 
}
\end{equation}

\subsection{Comparison of the perturbation results to numeric integration}
As a first example we use the parameters $\varepsilon=0.1$ and $y_0=1$ and compare the analytical and numerical solutions in the interval $400<\tau<500$. While the numerical results have been obtained by use of a Runge-Kutta 4 integration of (\ref{eq:y_simp}), the analytical result comes from evaluating the above perturbation theoretic computations to third order in $\varepsilon$.
Visualisation of the results has been restricted to the end of the integration interval, since in this region we expect the largest  differences between analysis and numeric integration. However, we observe a perfect agreement between the analytical (red) and numerical solutions (green), which is consistent with (\ref{eq:validity}) , leading to a range of validity of $\tau_\ast<1920$ which lies far above the integration interval of $\tau=[0,500]$. In Fig. 3 we plot the numeric solution (red) together with the perturbation theoretic result (green).

\begin{figure}[H]
	\centering
	  \includegraphics[width=\textwidth]{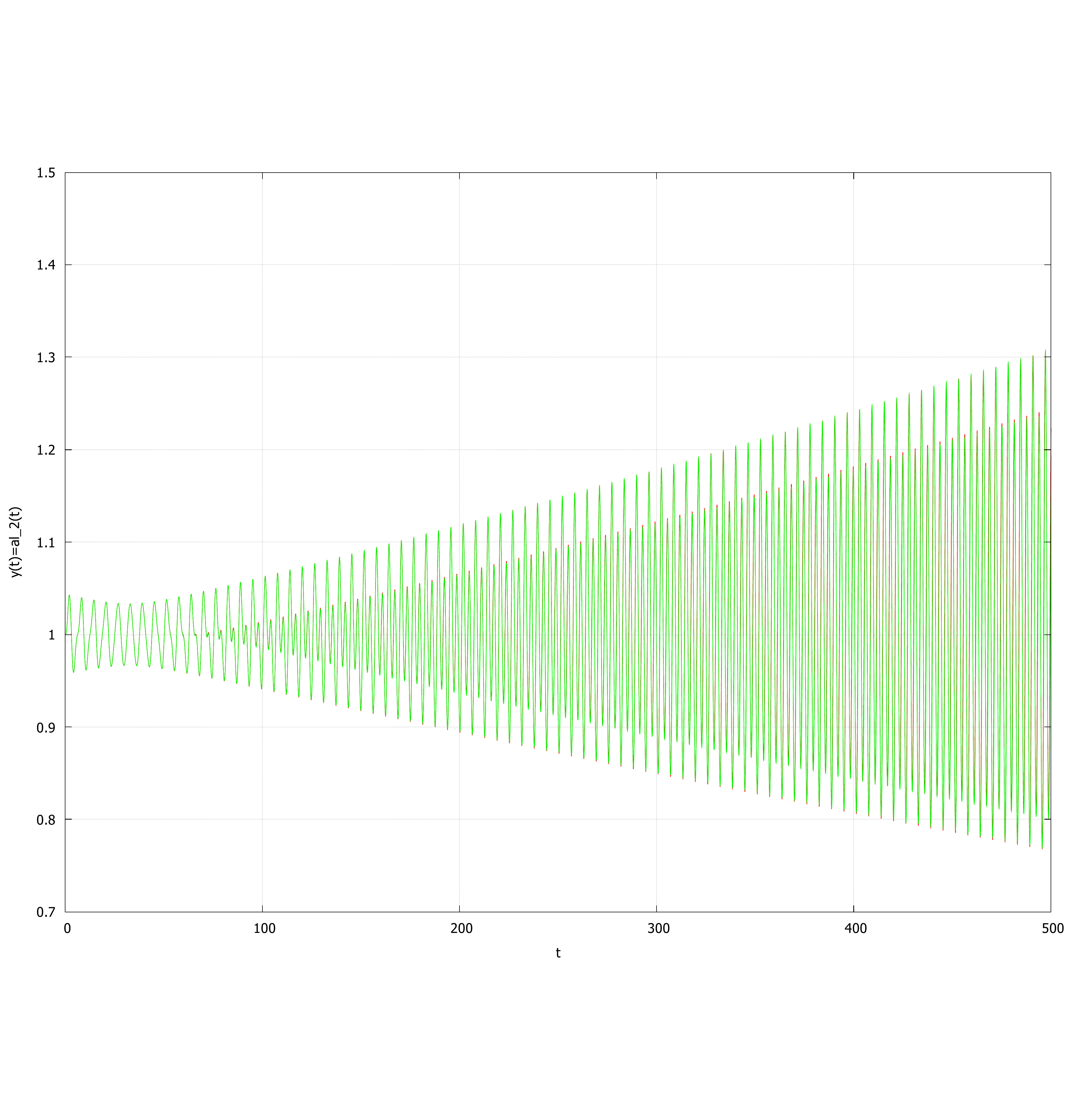}
	  \caption{Numeric versus analytic solution of (4.6) for $y_0=1$, $\varepsilon=0.1$.}
\end{figure}

As can be seen, we have an an excellent agreement of the analytical result and the numerical integration for $y_0=1$ and $\epsilon=0.1$. This is consistent with the theoretical validity region computed in (\ref{eq:validity})  which reaches up to  $\tau_\ast=1920$, well above the integration interval used. However, decreasing $y_0$ to $0.7$ the agreement between numerical and analytic computations starts to become less accurate. In the first interval $0<t<300$ shown in Fig. 4 the agreement is still reasonable up to about $\tau=200$, the theoretical limit being $\tau_\ast=225$. Above this limit the functions start to differ significantly. This becomes clearly visible in Fig. 5 where the numeric and analytic solutions differ by nearly $50\%$ at $\tau=500$.
\begin{figure}[H]
    \includegraphics[width=\textwidth]{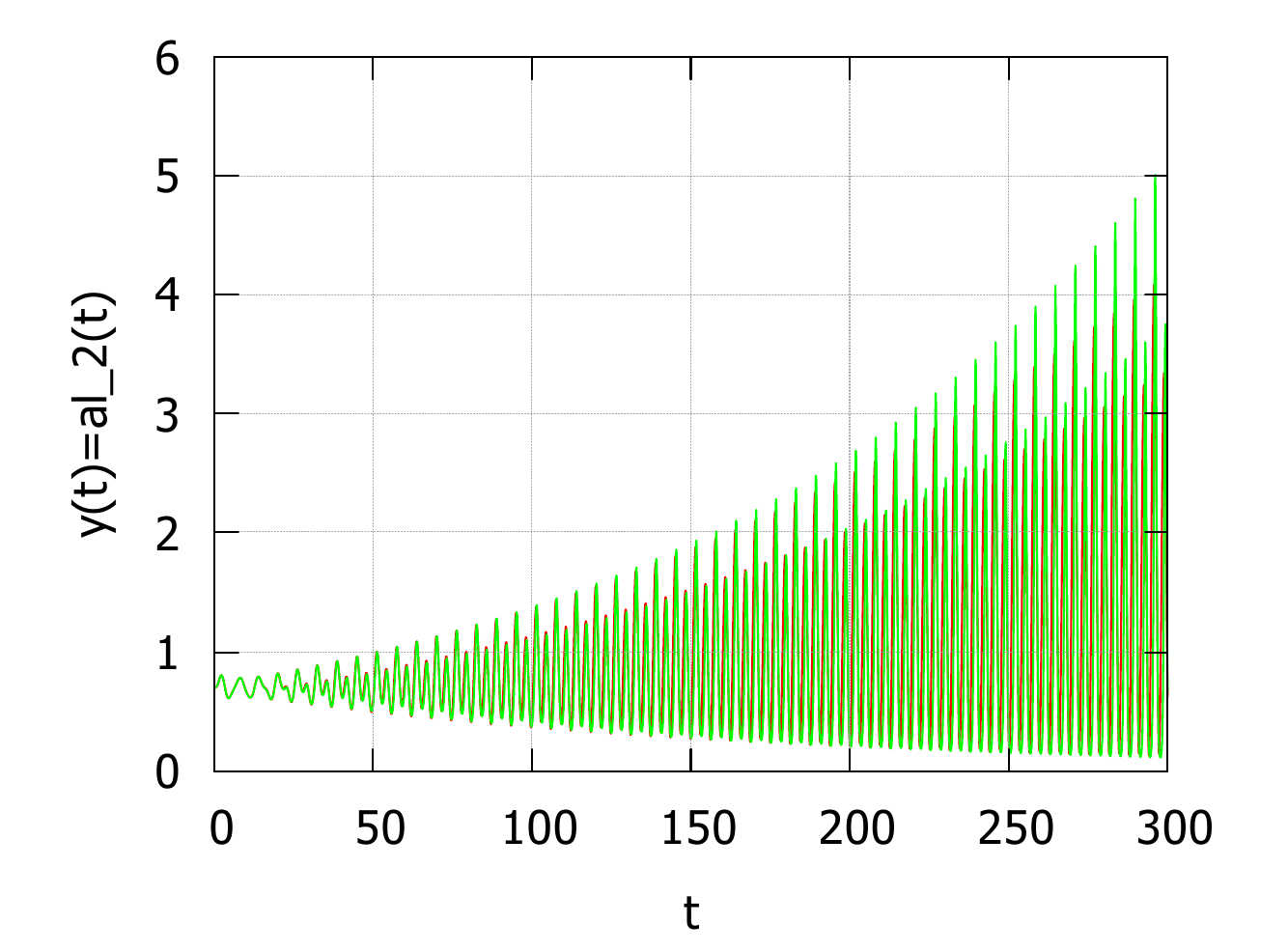}
    \caption{Numeric versus analytic solution when $y_0=0.7$, $\varepsilon=0.1$ and $0<t<300$}
\end{figure}
	
\begin{figure}[H]
	\includegraphics[width=\textwidth]{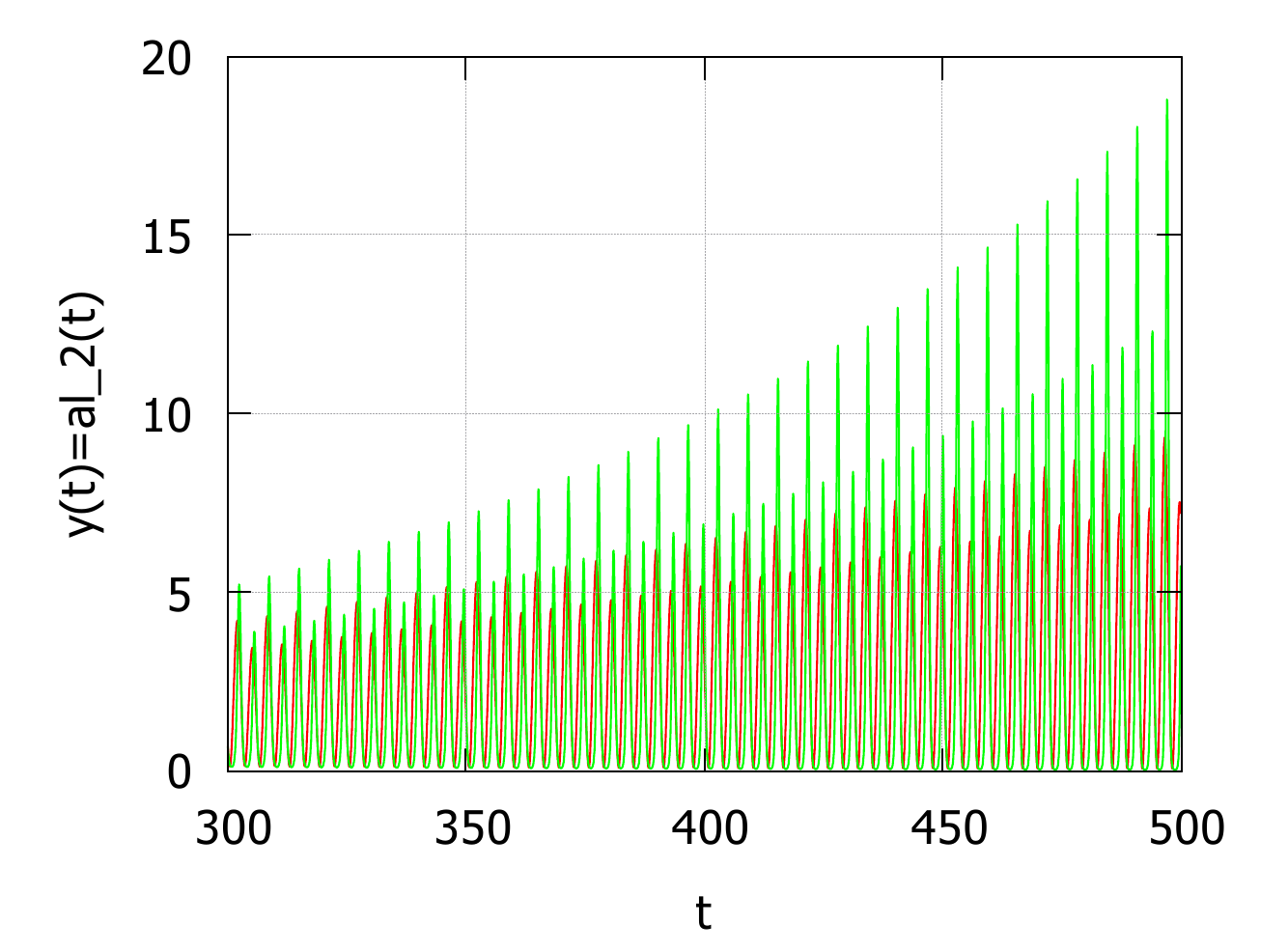}
	\caption{Numeric versus analytic solution when $y_0=0.7$, $\varepsilon=0.1$ and $300<t<500$}
\end{figure}
 
Note that for all graphs, the positivity property of $y$ is still strictly preserved due to consequently keeping the ansatz $y=e^{\rho(\tau)}$ throughout the computations. 

In conclusion, the good agreement between the analytical and numerical results 
within the predicted interval $\tau < \tau_\ast$ provides a clear confirmation of 
criterion~(5.13). The perturbative expansion up to $\mathcal{O}(\varepsilon^3)$ 
accurately reproduces the dynamics of~(4.6) as long as the secular term remains 
sufficiently small, while the onset of deviation beyond $\tau_\ast$ marks the natural limit of 
the perturbative validity. This demonstrates that the analytical approach not only 
captures the qualitative structure of $y(\tau)$ but also yields quantitatively 
reliable results within its theoretical window of applicability.

\section{Evaluation of the results}
In order to keep the analytic expressions simple as possible, for the evaluation of results we will restrict to the case $\omega=1$. In this case $\tau=t$ and $C_1=\varepsilon$. As before we keep $C_2=0$. So we inspect the invariant I related to 
\begin{equation}
	\label{eq:zom1}
	z''+z+g(t)z^2=0 \; ; \; g(t)=\alpha_2(t)^{-\frac{5}{2}} 
\end{equation}
In the following subsection we compute the analytical expression for the invariant $I$ based on the perturbation expansion of $\alpha_2(t)=y(t)$ to third order in $\varepsilon$:
\begin{equation}
\alpha_2(t)=y_0e^{\varepsilon\rho_1(t)+\varepsilon^2\rho_2(t)+\varepsilon^3\rho_3(t)+O(\varepsilon^4)}
\end{equation}

\subsection{Analytical approximation of the invariant $I(z,p,t$)}
We start from the invariant as given in (S4) with $\omega=1$, $C_1=\varepsilon$ and $C_2=0$.
\begin{align}	
\label{eq:invariant}
I(z,p,t) =  \alpha_2(t)p^2 - \alpha_2'(t)zp + \alpha_1(t)p 
 +\Big(\alpha_2(t)+\tfrac{1}{2}\alpha_2''(t)\Big)z^2 \\
 - \alpha_1'(t) z + \tfrac{2}{3}\alpha_2(t) g(t) z^3 \;=\; K 	\nonumber 
\end{align}	

From the perturbative treatment:
\begin{equation}
	\alpha_2(t)=y_0e^{\varepsilon\rho_1(t)+\varepsilon^2\rho_2(t)+\varepsilon^3\rho3(t)}
\end{equation}
follows
\begin{equation}
	\alpha_2'(t)=\alpha_2(t)\left(\varepsilon\rho_1'(t)+\varepsilon^2\rho_2'(t)+\varepsilon^3\rho3'(t)\right)
\end{equation}
In principle we could derive $\alpha_2''(t)$ from $\alpha'(t)$ by direct differentiation but this would generate an inacurracy due to amplification of numerical errors in truncated expansions while differentiating. Instead we derive $\alpha_2''(t)$ directly from the Volterra equation (\ref{eq:y_pert}) and get
\begin{equation}
	\alpha_2''(t)=4[y_0-\alpha_2(t)]+\varepsilon \int_0^t \alpha_2(s)^{-\frac{5}{2}} \cos{s} ds
\end{equation}
Using this method we replace differentiation of truncated expressions by integration which avoids numerical error amplification. Expanding the integrand w.r.t. $\varepsilon$ finally the invariant (\ref{eq:invariant}) is given by: \begin{equation}
	I(z,p,t)=A_1(t)z+A_2(t)p+A_3(t)z^2+A_4(t)zp+A_5(t)p^2+A_6(t)z^3 \label{eq:Ianalytic}
\end{equation}
with
\begin{equation}
	A_1(t)=-\alpha_1'(t)=-\frac{1}{2}\varepsilon \sin{t}
\end{equation}
\begin{equation}
	A_2(t)=\alpha_1(t)=\frac{1}{2}\varepsilon \cos{t}
\end{equation}
\begin{equation}
A_3(t)= \alpha_2(t)+\frac{1}{2} \alpha_2''(t) = A_{31}(t)+y_0e^{A_{32}(t)}
\end{equation}
where

\begin{align}
	A_{31}(t)
	&=
	\frac{\varepsilon^3}{6912y_0^{\frac{19}{2}}} \left(-25\sin{5t}+120\sin{4t}-531\sin{3t}+644\sin{2t} \right. \nonumber \\ 
	&\qquad
	 \left. -230\sin{t}+135t\cos{3t}+120t\cos{2t}-75t \cos{t}\right) \\
	&\qquad 
		+\frac{\varepsilon^2}{144y_0^6}\left(-9\cos{3t}+4\cos{2t}+5\cos{t}\right)+
		\frac{\varepsilon}{6y_0^{\frac{5}{2}}}\left(2\sin{2t}-\sin{t}\right) \nonumber
\end{align}

\begin{align}
	A_{32}(t)
	&=
	\frac{\varepsilon^3}{2592y_0^{\frac{21}{2}}} \left(\sin{6t}-\frac{9}{4}\sin{5t}+\frac{69}{8}\sin{4t}+\frac{73}{4}\sin{3t}-\frac{327}{4}\sin{2t}+ \right.   \nonumber \\ 
	&\qquad \left. \frac{159}{2}\sin{t}
	 -\frac{45}{4}t\cos{4t}+\frac{45}{4}t\cos{3t}-\frac{45}{2}t\cos{2t}+\frac{135}{4}t\cos{t}-\frac{45}{4}t \right) \nonumber \\
	&\qquad 
	+\frac{\varepsilon^2}{y_0^7}\left(\frac{1}{144}\cos{4t}-\frac{1}{72}\cos{3t}+\frac{19}{288}\cos{2t}-\frac{1}{24}\cos{t} \right.  \\  
	c&\qquad
	\left. +\frac{5}{96}t\sin{2t}-\frac{5}{288} \right)+\frac{\varepsilon}{y_0^\frac{7}{2}}\left(\frac{1}{3}\sin{t}-\frac{1}{6}\sin{2t} \right) \nonumber
\end{align}

\begin{align}
	A_4(t)
	&=
	-\frac{\varepsilon^3}{3456y_0^{\frac{19}{2}}} \left( 5cos{5t}-30\cos{4t}+192\cos{3t}-292\cos{2t}+155\cos{t} \right. \nonumber \\ 
	&\qquad
	\left. +45t\sin{3t}+60t\sin{2t}-75t\sin{t}-30 \right) \\
	&\qquad 
	-\frac{\varepsilon^2}{288y_0^6}\left(-12\sin{3t}-7\sin{2t}+20\sin{t}+30t\cos{2t} \right) \nonumber \\  
	&\qquad
	+\frac{\varepsilon}{3y_0^\frac{5}{2}}\left(\cos{t}-\cos{2t} \right) \nonumber
\end{align}

\begin{align}
	A_5(t)
    &=
    \alpha_2(t)=y_0e^{\rho(t)}=y_0e^{\rho(t)}
\end{align}
with $\rho_n(t)$ given in (\ref{eq:rho1-sol}) - (\ref{eq:rho3_sol}).
\begin{equation}
	A_6(t)=\frac{2}{3}\alpha_2(t)g(t)=\frac{2}{3}\alpha_2(t)^{-\frac{3}{2}}=\frac{2}{3}y_0e^{-\frac{3}{2} \rho(t)}
\end{equation}
In order to check for the invariant property of (\ref{eq:Ianalytic}) we numerically integrate the associated nonlinear equation (\ref{eq:zom1}) with initial conditions $z_0=0.2$ , $z'(0)=p_0=0$. over the time interval $0<t<500$ and use the successive values at points $t_n=n\Delta t$ as well as $t_n$ to be inserted into the analytic expression (\ref{eq:Ianalytic}). We then compute the deviation (in $\%$) from the initial value of $I(t)$. In theory this value should be equal to zero but due to the truncation of the perturbative solution, we expect a remaining time variation. We have two free parameters to our disposal, namely $\varepsilon$ and $y_0$, leading to coefficient functions $g_{\varepsilon,y_0}(t)$ entering (\ref{eq:zom1}).
We fix $\varepsilon=0.05$ and use four values for $y_0$, namely $y_0=1.2 , 1.1 , 0.9 , 0.8$.  Inspecting the analytic expressions (\ref{eq:rho1-sol}), (\ref{eq:rho2-sol}) and  (\ref{eq:rho3_sol})	it becomes clear, that this behaviour arises from increasingly high powers of $y_0$ in the denominators of the perturbation contributions which attenuate high order perturbation contributions for $y_0>1$ as it strengthens them when $y_0<1$. Infact in the four figures below we observe the predicted behaviour. While for $y_0>1$ the invariant shows no visible variation (infact in the tenth of $\%$ regime) this changes abruptly as $y_0<1$, so that for $y_0=0.8$ the variation of the approximation reaches $8\%$. 

\noindent {\bf It is important to mention, that these variations are not due to  the fact of the integral itself just being valid locally, but only to a vanishing convergence of the perturbation series in the regime $y_0<1$!}  

Finally we present a figure of the coefficient function $g(t)=\alpha_2(t)^{-\frac{5}{2}}$ rendering (\ref{eq:zom1})  integrable, using $y_0=1.1$ and $\varepsilon=0.05$.  For low values of $t$ it behaves like a damped oscillation reaching an amplitude minimum and then starts growing indefinitely while exhibiting a superposition of alt least two frequencies ($T=\pi$ and $2\pi$) with linearly increasing amplitude.

\begin{figure}[H]
	\centering
	\begin{subfigure}[t]{0.48\linewidth}
		\centering
		\includegraphics[width=\linewidth]{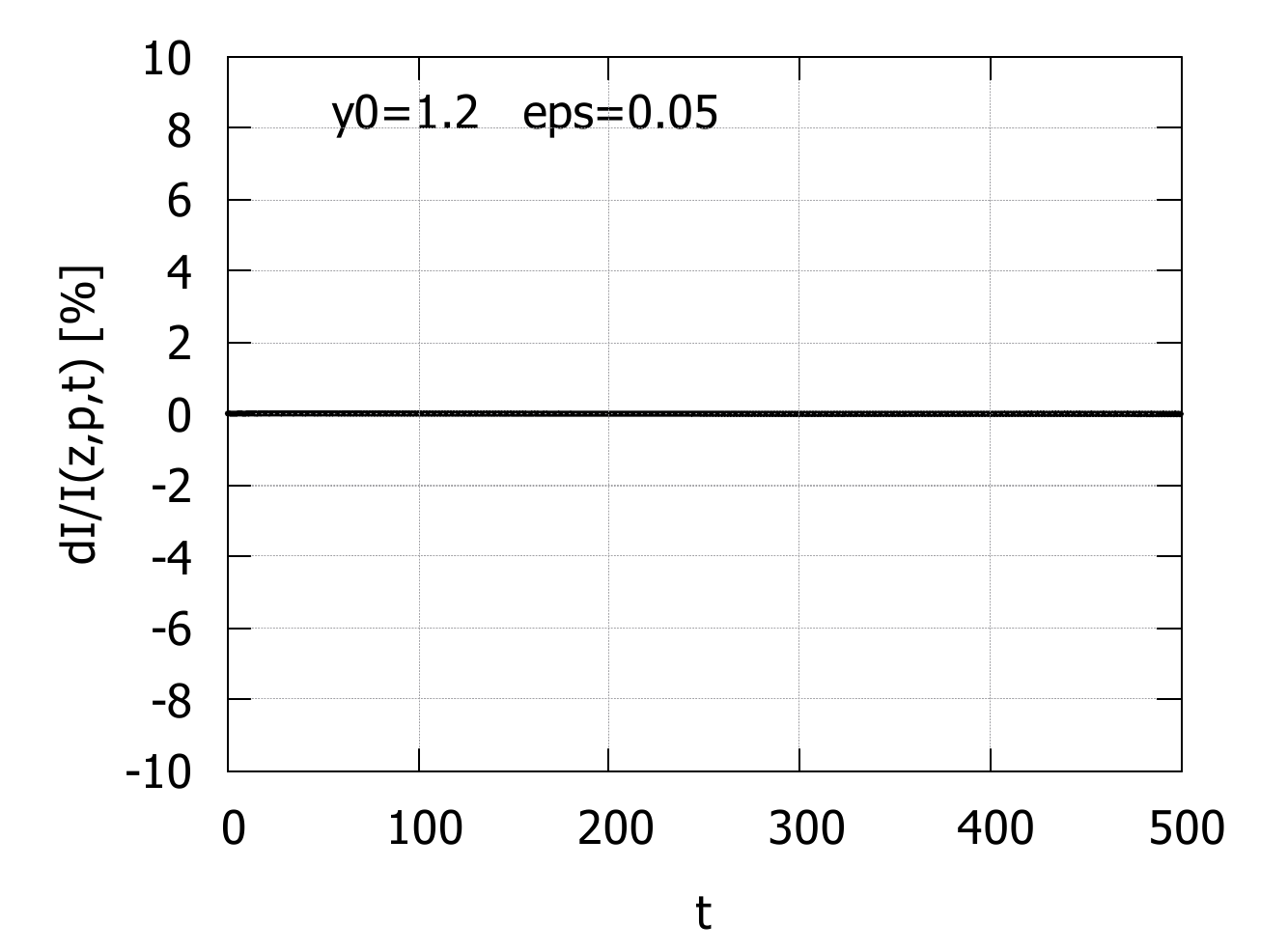}
		\caption{$y_0=1.2,\ \varepsilon=0.05$ — no variations visible.}
		\label{fig:tube-y0-1p2}
	\end{subfigure}\hfill
	\begin{subfigure}[t]{0.48\linewidth}
		\centering
		\includegraphics[width=\linewidth]{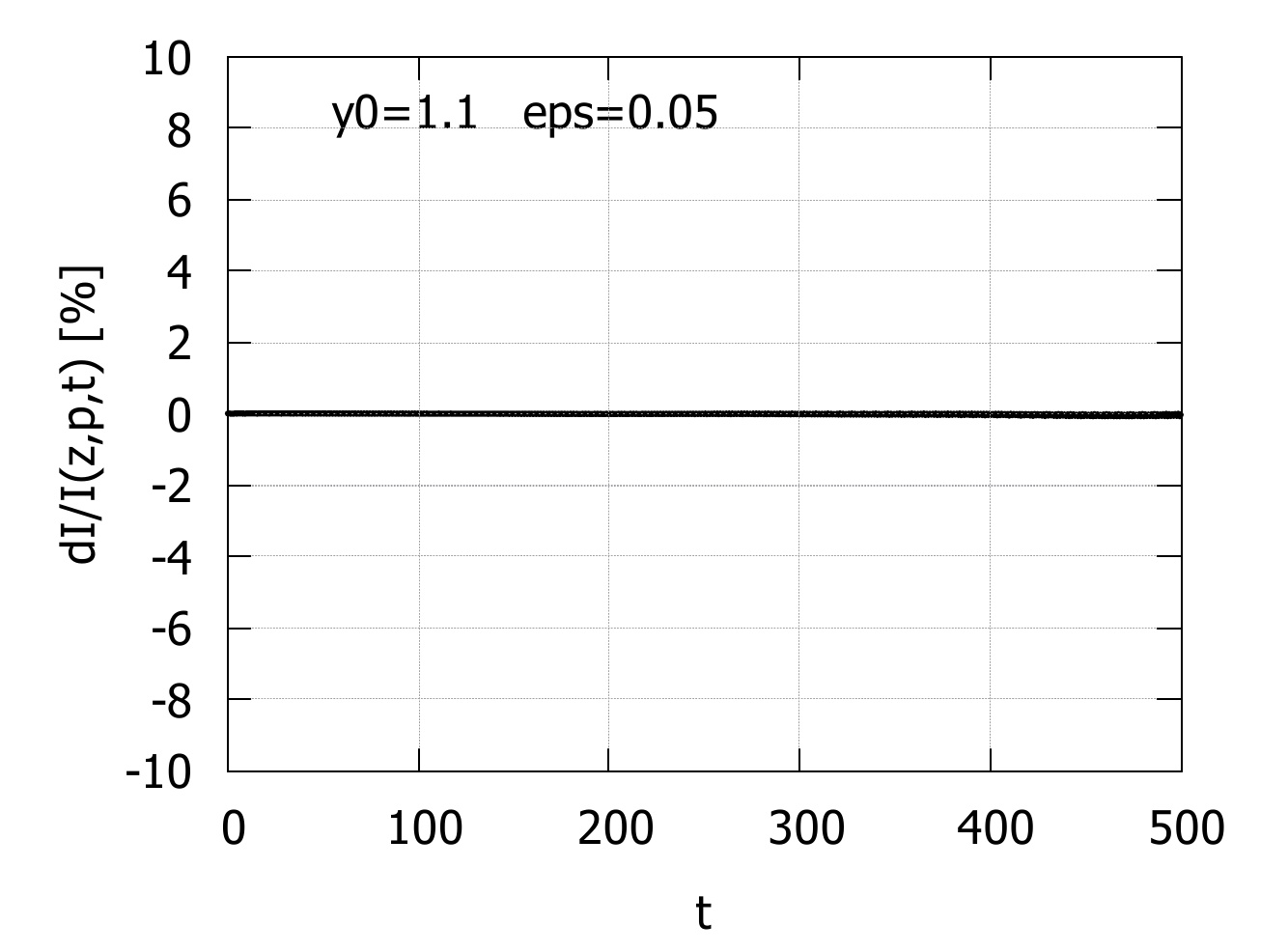}
		\caption{$y_0=1.1,\ \varepsilon=0.05$ — very small variations visible.}
		\label{fig:tube-y0-1p1}
	\end{subfigure}
	
	\vspace{0.8em}
	
	\begin{subfigure}[t]{0.48\linewidth}
		\centering
	\includegraphics[width=\linewidth]{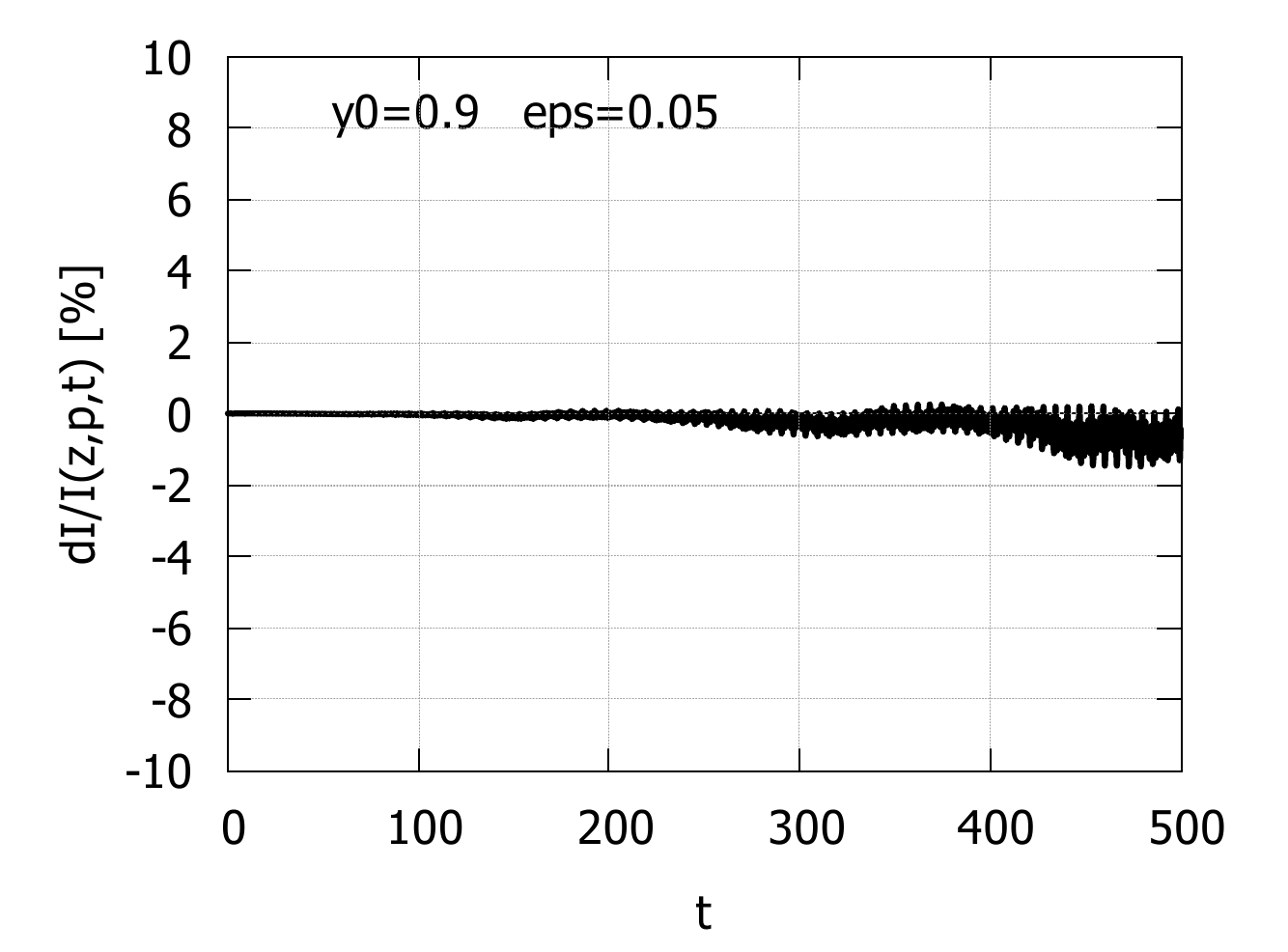}
		\caption{$y_0=0.9,\ \varepsilon=0.05$ — A variation of about 2$\%$ becomes visible}
		\label{fig:tube-y0-0p9}
	\end{subfigure}\hfill
	\begin{subfigure}[t]{0.48\linewidth}
		\centering
		\includegraphics[width=\linewidth]{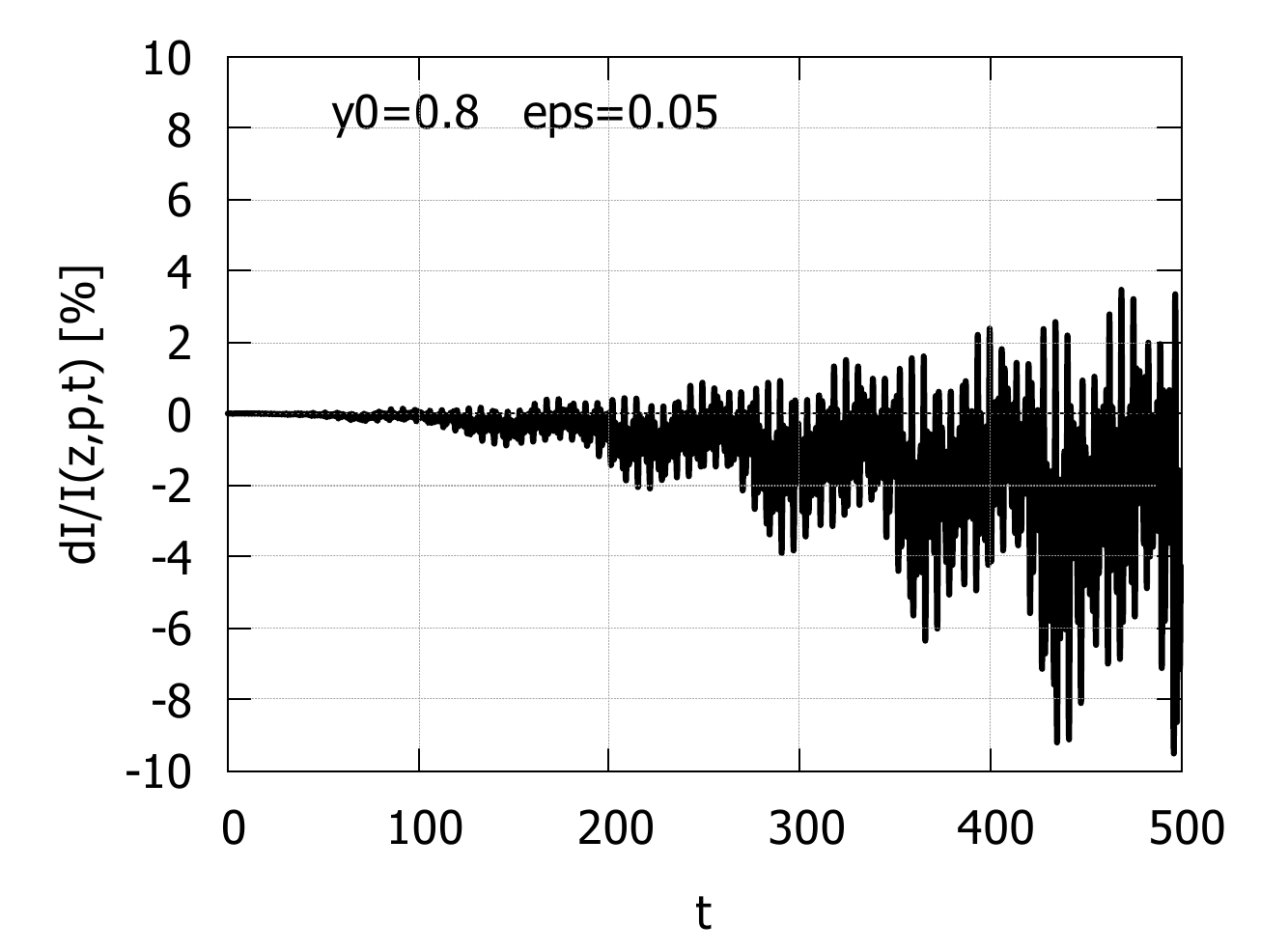}
		\caption{$y_0=0.8,\ \varepsilon=0.05$ — A strong variation of $9\%$ appears}
		\label{fig:tube-y0-0p5}
	\end{subfigure}
	
	\caption{Variations of the approximate analytical invariant $I(z,p,t)$ in percents of its initial value for $\varepsilon=0.05$ and $y_0=1.2 , 1.1 , 0.9 , 0.8$}
	\label{fig:tube-four-demos}
\end{figure}

\begin{figure}[H]
	\includegraphics[width=\textwidth]{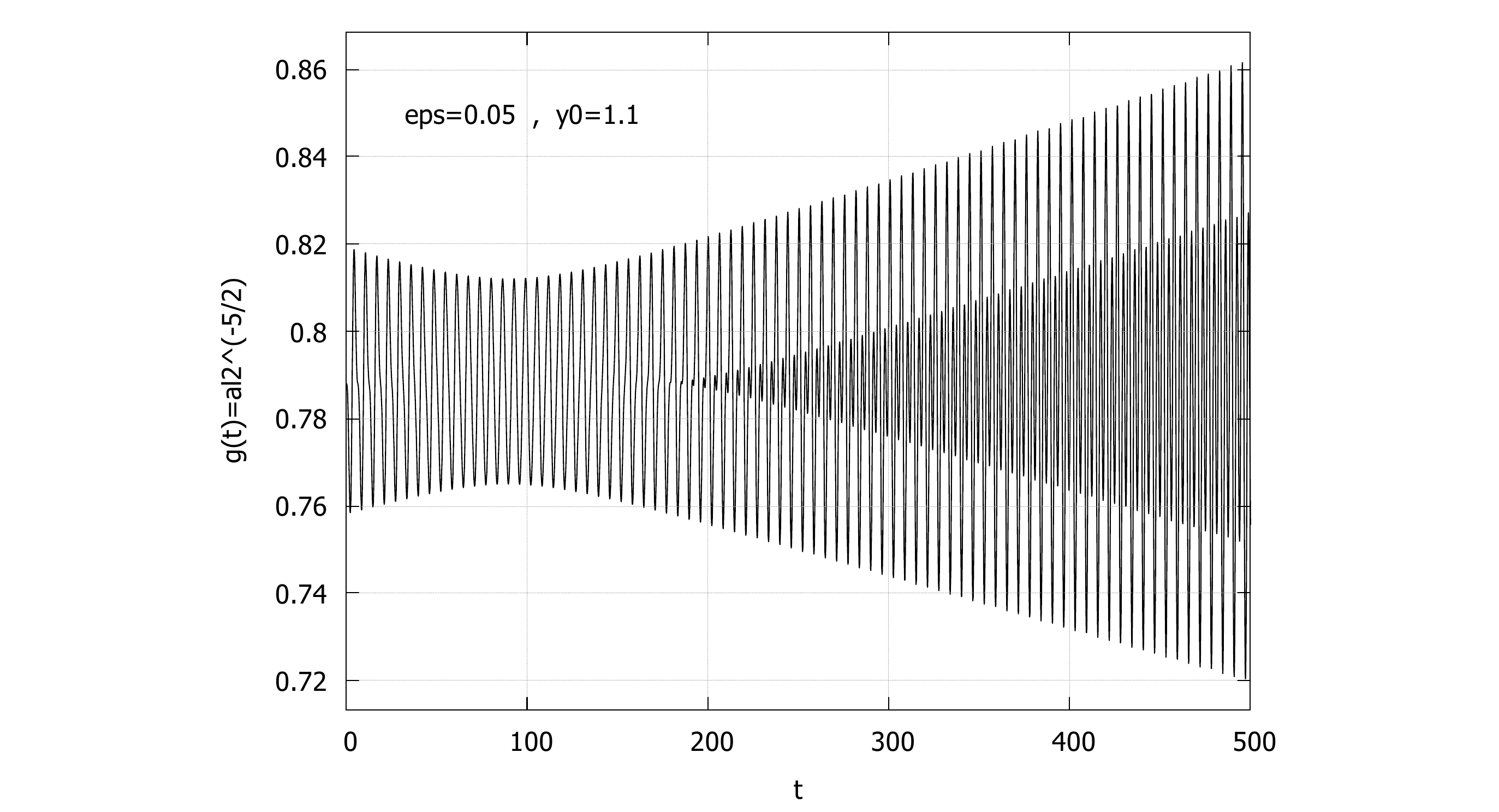}
	\caption{Graph of the coefficient function $g(t)$ when $\varepsilon=0.05$ and $y_0=1.1$}
\end{figure}

\section{Conclusions}
\label{sec:conclusions}

We have analyzed the nonlinear oscillator \eqref{osc_non_lin} with a time-dependent 
coefficient $g(t)$ and shown that it admits an exact quadratic invariant provided 
that $g(t)=\alpha_2(t)^{-5/2}$ and $\alpha_2(t)$ satisfies the nonlinear 
third-order equation derived in Section 2. This establishes a 
broad class of integrable non-autonomous systems.

A central finding is the qualitative distinction between two geometric regimes.
If $\alpha_2(t)$ is periodic, the invariant surfaces are compact and form
invariant tori, recovering the classical Liouville picture. However, we have
shown that periodic solutions of $\alpha_2(t)$ are generically obstructed by a
second-order resonance mechanism, which enforces a slow but unavoidable
aperiodic modulation. In this generic regime, the invariant surfaces become
non-compact: the motion remains strictly confined, but the trajectory winds
along a \emph{tube} rather than closing on a torus.

We have proposed the term \emph{tube integrability} to characterize this regime:
the system retains a global invariant and regular geometric structure, but the
associated invariant manifolds are non-compact in the time direction. This
extends the conceptual scope of integrability beyond the classical assumption
of compactness.

To understand the behavior of $\alpha_2(t)$ in the non-periodic setting, we
developed a perturbation expansion up to third order and verified its accuracy
against numerical integration. The effective small parameter was shown to be
$\varepsilon_{\mathrm{eff}} = \varepsilon\, y_0^{-7/2}$, explaining the transition
between quantitatively stable and unstable approximation regimes.

\medskip
\noindent\textbf{Crucially,} the breakdown of accuracy for $y_0 < 1$ does \emph{not}
indicate a loss of integrability. The invariant $I(z,p,t)$ remains exact for all $t$.
What breaks down is only the truncated asymptotic expansion of $\alpha_2(t)$,
whose coefficients are governed by nearby singularities in the complex 
$\varepsilon$–plane. Thus the observed drift of $I(z,p,t)$ in this regime reflects 
the natural limit of the perturbation series, not a transition to chaotic dynamics.

\medskip
In summary, this work identifies and characterizes a coherent form of
integrability beyond periodicity, in which the invariant structure is
non-compact but fully regular. This raises several natural directions for
future investigation, including the onset of tube-chaotic transitions,
Lyapunov analysis of the $\alpha_2(t)$ dynamics, and the classification of
non-compact invariant geometries in higher-dimensional systems.

\section*{Acknowledgments}

The author would like to thank Serge Bouquet (CEA, France) for many discussions 
on the special, torus-integrable case $\varepsilon=0$, which laid the foundation for the present 
analysis. I am also grateful to Christoph Lhotka (University of Rome) for 
many years of continuous scientific exchange, and to Gilbert Guignard (CERN) 
for introducing me to nonlinear particle dynamics and for numerous insights 
into the geometric interpretation of invariants. 

A special acknowledgment is due to my former mathematics teacher 
Heribert Hartmann, whose early guidance made it possible for me to pursue 
mathematics at all.

Finally, I would like to acknowledge the assistance provided by 
\textbf{OpenAI's GPT-5 model}, whose contributions in structuring, clarifying 
and refining the presentation were of significant help during the preparation 
of this work.

\end{document}